\documentclass[12pt, 14paper,reqno]{amsart}
\usepackage{amsmath,amsfonts,amssymb}

\theoremstyle{plain}
\newtheorem{theorem}{Theorem}[section]
\newtheorem{lemma}[theorem]{Lemma}

\newtheorem{proposition}[theorem]{Proposition}
\theoremstyle{proof}
\theoremstyle{definition}

\theoremstyle{remark}

\theoremstyle{lamma}

\usepackage{amsmath}
\usepackage{amsfonts}   
\usepackage{amssymb}
\usepackage{amssymb, amsmath, amsthm}
\usepackage[breaklinks]{hyperref}
\theoremstyle{thmrm}

\usepackage{graphicx}
\begin{document}
\title[Class numbers of cyclotomic fields]{On the plus parts of the class numbers of cyclotomic fields}
\author{Kalyan Chakraborty and Azizul Hoque}
\address{Kalyan Chakraborty @Harish-Chandra Research Institute, A CI of Homi Bhabha National Institute, Chhatnag Road, Jhunsi- 211019, Allahabad, India.}
\email{kalychak@gmail.com}

\address{Azizul Hoque @Department of Mathematics, Faculty of Science, Rangapara College, Rangapara, Sonitpur-784505, Assam, India.}
\email{ ahoque.ms@gmail.com}

\keywords{Class numbers, Maximal real subfield of cyclotomic fields, Real quadratic fields}
\subjclass[2010] {Primary: 11R29, 11R18, Secondary: 11R80}
\maketitle

\begin{abstract}
We exhibit some new families of cyclotomic fields which have non-trivial plus parts of their class numbers. We also prove the $3$ - divisibility of the plus part of the class number of another family consisting of infinitely many cyclotomic fields.  At the end, we provide some numerical examples supporting our results.
\end{abstract}

\section{Introduction}
Let $\zeta_m$ be a primitive $m$-th root of unity for a positive integer $m$, then the field $K^+_m=\mathbb{Q}(\zeta_m + \zeta^{-1} _m)$ is the maximal real subfield of the cyclotomic field $K_m=\mathbb{Q}(\zeta_m)$. Let $\mathcal{H}^+(m)$ denote the class-number of  $K^+_m$ and $h(m)$ be that of $k_m= \mathbb{Q}(\sqrt{m})$. The class number $\mathcal{H}(m)$ of $K_m$ can be written as $\mathcal{H}(m)=\mathcal{H}^+(m)\mathcal{H}^-(m)$. The factor $\mathcal{H}^-(m)$, so called relative class number, is well understood and is usually rather large. This factor can be determined in terms of Bernoulli numbers using the complex analytic class number formula (for details see, pp. 79-84 in \cite{LA78}). Earlier in 1850,
E. E. Kummer \cite{KU50, KU51} computed $\mathcal{H}^-(m)$ for all primes $m$ upto $97$. One can find $\mathcal{H}^-(m)$ from the tables given in \cite{SC64} by G. Schrutka von Rechtenstamm for any positive integer $m$ satisfying $\phi (m)<256$, where $\phi$ stands for Euler's phi function,  In 1998, R. Schoof \cite{SC98} computed $\mathcal{H}^-(m)$ for all odd primes $m<509$ and in fact, he also gave  the structure of the corresponding class groups. 

On the other hand, the factor $\mathcal{H}^+(m)$ is not well understood and is notoriously hard to compute explicitly. In this case, the complex analytic class number formula is not so useful, since it appeals that the units of $K^+_m$ to be known. Till the date, there is no useful method to compute $\mathcal{H}^+(m)$, not even for relatively small $m$. The number $\mathcal{H}^+(m)$ is known only for all primes up to $151$ and extended up to $241$ under the assumption of GRH (generalized Riemann hypothesis). More precisely, J. C. Miller \cite{MI15} proved that $\mathcal{H}^+(p)=1$ for all prime $p\leq 151$ unconditionally. In the same paper, he also proved that $\mathcal{H}^+(p)=1$ for all primes $p\leq 241$ except $p=163, 181, 229$ for which $\mathcal{H}^+(p)$ is $4, 11, 3,$ respective (again assuming GRH). For these primes, the Kummer--Vandiver conjecture, which says that $m$ does not divide $\mathcal{H}^+(m)$ if $m$ is a prime, holds. Recently, J. P. Buhler and D. Harvey \cite{BH11} confirmed this conjecture for all primes less than 163577856. 

On the other hand, N. C. Ankeny, S. Chowla and H. Hasse \cite{ACH65} proved that $\mathcal{H}^+(p)>1$ if $p=(2n q)^2+1$ is a prime, where $q$ is a prime and $n>1$ is an integer. Subsequently, S. D. Lang \cite{LA77} proved that $\mathcal{H}^+(p)>1$ for any prime of the form $p=\lbrace (2n+1)q \rbrace ^2 + 4$, where $q$ is a prime and $n\geq 1$ is an integer. In 1987, H. Osada \cite{OS87} generalized both the results. More precisely, he proved that $\mathcal{H}^+(m)>1$ if $m=(2n q)^2 +1$ is a square-free integer, where $q$ is a prime and $n$ is a positive integer such that $n\neq 1, q$. In the same paper, he also proved if $m=\lbrace (2n+1)q\rbrace ^2 +4$ is a square-free integer, where $q$ is a prime and $n$ is a positive integer such that $n\neq q$, then $\mathcal{H}^+(m)>1$. All these results have been obtained in the case $m\equiv 1 \pmod 4$.

Furthermore, H. Takeuchi \cite{TA81} established  similar results for certain primes of the form $p\equiv 3 \pmod 4$. He proved that if both $12m+7$ and $p=\{3(8m+5)\}^2-2$ are primes, where $m\geq 0$ is an integer, then $\mathcal{H}^+(4p)>1$. In the same paper, he also proved that if both $12m+11$ and $p=\{3(8m+7)\}^2-2$ are primes, where $m\geq 0$ is an integer, then $\mathcal{H}^+(4p)>1$. Recently, A. Hoque and H. K. Saikia \cite{HS16} generalized both the results. More precisely, they proved that if $m=\{3(8g+5)\}^2-2$ is a square-free integer, where $g$ is a positive integer, then $\mathcal{H}^+(4m)>1$. In the same paper, they also obtained a similar result for any square-free integer $m=\{ 3(8g+7)\}^2-2$, where $g$ is a positive integer. Along the same line, we \cite{HC17} produced some interesting families of cyclotomic fields whose maximal real subfields have class numbers bigger than one. All these results have been obtained in case when $m\equiv 3 \pmod 4$. Thus it would be interesting to try exhibiting similar families depending on $m$ where $m \equiv 2 \pmod 4$. Here we find some families of cyclotomic fields $K_{4m}$ whose maximal real subfields $K^+_{4m}$ have non-trivial class number when $m\equiv 2 \pmod 4$.

We discuss the results in three sections. In \S 2, we produce some families of cyclotomic fields $K_{4m}$ with maximal real subfields $K^+_{4m}$ having non-trivial class numbers whenever $m\equiv 2 \pmod 4$. In \S 3, we discuss the divisibility of the class numbers of maximal real subfields of a class of cyclotomic fields. More precisely, we produce a family of cyclotomic fields whose maximal real subfields have class numbers a multiple of $3$. 
In the concluding section, we provide some numerical evidence of our results.
We have used PARI 2.9.1 \cite{PA} for these computations.    

\section{ Non-triviality of $\mathcal{H}^+(m)$}
We prove some results concerning the non-triviality of  class numbers of certain maximal real subfields of cyclotomic fields. The proofs use elementary techniques  on dealing with solutions of Diophantine equations, and basic properties of quadratic and cyclotomic fields. We begin with a family of real quadratic fields and show that they have non-trivial class numbers.
\begin{proposition}\label{proposition3.1}
Let $m=\{14(2n+1)\}^2+2$ be square-free with $n$ a positive integer. Then $h(m)>1$. 
\end{proposition}

\begin{proof}
We observe that
$$
m=\{14(2n+1)\}^2+2\equiv 2\pmod 7.
$$
Therefore the residue symbol, $\big( \frac{m}{7} \big) = \big( \frac{2}{7} \big)=1$. Thus 
$7$ splits completely in $k_m=\mathbb{Q}(\sqrt{m})$ as a product of a prime ideal $\mathfrak{A}\subset \mathcal{O}_{k_m}$ and its conjugate $\mathfrak{A}'$ with absolute norm $N_{k_m}(\mathfrak{A}) =7$. 

Let us assume that $h(m)=1$. Then $\mathfrak{A}$ is principal and thus, since $m\equiv 2\pmod 4$, we can write 
$$
\mathfrak{A}=(a+b\sqrt{m})
$$
with $a, b\in \mathbb{Z}$. 
Therefore, we have (using $|N_{k_m}(\mathfrak{A})|=7$),
\begin{equation*}
a^2-mb^2=\pm 7.
\end{equation*}
We prove that such a Diophantine equation doesn't have any rational integer solution.
Let us consider
\begin{equation}\label{eq3.1}
a^2-mb^2=7,
\end{equation}
and let $t=392n^2+392n+99$. Then $t\equiv 1\pmod 7$ and $m=2t$. Thus (\ref{eq3.1}) implies that
\begin{equation}\label{eq3.2}
a^2\equiv 7\pmod t.
\end{equation}
But, quadratic reciprocity law gives, 
$$
\bigg( \frac{7}{t}\bigg)=(-1)^{3(196n^2+196n+49)}\bigg(\frac{1}{7}\bigg)=-1
$$
which contradicts (\ref{eq3.2}). 

We now look at the other case, i.e.
\begin{equation*}\label{eq3.3}
a^2-mb^2=-7.
\end{equation*}
Let $a_0$ be any integer and $b_0$ be the least positive integer such that 
\begin{equation}\label{eq3.3}
a_0^2-mb_0^2=- 7.
\end{equation}
Writing (\ref{eq3.3}) in norm form, we have
\begin{equation}\label{eq3.4}
N_{k_m}(\pm |a_0|+b_0\sqrt{m})=-7.
\end{equation}
We now multiply (\ref{eq3.4}) with the norm of the unit, 
$$
\epsilon_m=\mp (s^2+1)+s\sqrt{m},
$$
where $s=14(2n+1)$, in the field $k_m$. Then we get 
\begin{equation*}
(-|a_0|(s^2+1)+b_0sm)^2-(\pm b_0(s^2+1)\mp |a_0|s)^2m=-7.
\end{equation*}
Using now the minimality of $b_0$, we can write,
\begin{equation*}
|b_0(s^2+1)-|a_0|s|\geq b_0.
\end{equation*}
If $b_0(s^2+1)-|a_0|s\geq b_0$, then $|a_0|\leq b_0s$, and therefore (\ref{eq3.3}) gives that $2b_0^2<7$ and that would imply that $b_0=1$. 

Now if $b_0=1$, the relation (\ref{eq3.3}) implies that
$a_0^2\equiv 3 \pmod 4$. This is an absurd.

In the other case; $|a_0|s-b_0(s^2+1)\geq b_0$ would give $|a_0|s > b_0m$ and then (\ref{eq3.3}) leads to $2mb_0^2<-7s^2$. This is not possible as $m>0$.
\end{proof}
The following Lemma which can be derived from a result (main theorem, \cite{OS89}) of H. Osada, will be of our use.
\begin{lemma}\label{lemma1}
Let $m$ be a square-free positive integer. Then the ideal class group of $K^+_{\sigma_m^2 m}$ has a subgroup which is isomorphic to $\mathcal{C}(k_m)^2$, where $\mathcal{C}(k_m)$ is the ideal class group of $k_m=\mathbb{Q}(\sqrt{m})$, and 
\begin{equation}\nonumber
\sigma_m=
\begin{cases}
1 \text{ if } m\equiv 1\pmod4,\\
2 \textit{ if } m\equiv 2, 3\pmod 4.
\end{cases}
\end{equation}
\end{lemma}
We thus obtain the first main result by applying together Proposition \ref{proposition3.1} and Lemma \ref{lemma1}.
\begin{theorem}\label{theorem3.1}
Let $m=\{14(2n+1)\}^2+2$ be square-free with $n$ a positive integer. Then $\mathcal{H}^+(4m)>1$.
\end{theorem}
We exhibit another family of cyclotomic fields with  non-trivial plus parts in their class groups.
\begin{proposition}\label{proposition3.2}
Let  $m=(3(2n+1))^2+1$ be a square-free integer with $n\geq 1$. Then $h(m)>1$. 
\end{proposition}
\begin{proof} Let $m=(3(2n+1))^2+1$. Then 
$$
m\equiv 1\pmod 3,
$$
and thus 
$$
\bigg( \frac{m}{3} \bigg)=1.
$$
Thus we can write 
$$(3)=\mathfrak{B}\mathfrak{B}',\hspace*{.5cm} (\mathfrak{B}\ne \mathfrak{B}'),$$ 
where $\mathfrak{B}$ and $\mathfrak{B}'$ are prime (conjugates) ideals in $\mathcal{O}_{k_m}$ with $N_{k_m}(\mathfrak{B})=3$. 

Let us assume that $h(m)=1$. Then $\mathfrak{B}$ is principal and thus since $m\equiv 2\pmod 4$, $\mathfrak{B}$ can be expressed as 
$$\mathfrak{B}=(a+b\sqrt{m}) \text{ with }a, b\in \mathbb{Z}.$$ 
Therefore, we have
\begin{equation*}
a^2-mb^2=\pm 3.
\end{equation*}
Clearly, $b\ne 0$. Let us assume that $a_0$ be an integer, and let $b_0$ be the least positive integer such that 
\begin{equation}\label{eq3.5}
a_0^2-mb_0^2=\pm 3.
\end{equation}
Then $N_{k_m}(\alpha)=\pm 3$ for some integer $\alpha=a_0-b_0\sqrt{m}$.

Let us suppose $r=3(2n+1)$. Then the fundamental unit $\epsilon_m$ in $k_m$ is given by $$\epsilon_m=r+\sqrt{m}.$$

We now have $N_{k_m}(\alpha \epsilon_m)=\pm 3$ which implies that
\begin{equation*}
(a_0r-b_0m)^2-(a_0-b_0r)^2m=\pm 3.
\end{equation*}
Employing the minimality of $b_0$, 
\begin{equation*}
|a_0-b_0r|\geq b_0.
\end{equation*}
If $a_0-b_0r\geq b_0$, then $a_0\geq b_0(r+1)$ and thus (\ref{eq3.5}) gives that $2rb_0^2\leq \pm 3$. This is not possible as $r\geq 3$.

Again, if $b_0 r-a_0\geq b_0$, then $a_0\leq b_0(r-1)$. Thus from (\ref{eq3.5}) we observe that $b_0^2(r-1)^2-b_0^2m\geq \pm 3$ which implies that $-2rb_0^2\geq \pm 3$. 
This once again leads to an impossibility as $r\geq 3$. Thus we complete the proof.  
\end{proof}

We now use Proposition \ref{proposition3.2} and Lemma \ref{lemma1} to obtain the following:
\begin{theorem}\label{theorem3.2}
Let $m=(3(2n+1))^2+1$ with $n$ a positive integer. Then $\mathcal{H}^+(4m)>1$.
\end{theorem}
We provide another similar family of maximal real subfields of certain cyclotomic fields each with class number bigger than $1$.
\begin{proposition}\label{proposition3.3}
Let $m=\{6(2n+1)\}^2-2$ with $n\geq 1$ an integer. Then $h(m)>1$. 
\end{proposition}

\begin{proof}
We observe that
$$m=\{6(2n+1)\}^2-2\equiv 1\pmod 3.$$
Therefore $\big( \frac{m}{3} \big)=\big( \frac{1}{3} \big)=1$, and thus we have $$(3)=\mathfrak{C}\mathfrak{C}'$$
with $\mathfrak{C}\ne \mathfrak{C}'$, where $\mathfrak{C}$ and $\mathfrak{C}'$ are prime ideals in $\mathcal{O}_{k_m}$ with absolute norm $N_{k_m}(\mathfrak{C})=3$. 

Now if $h(m)=1$, then $\mathfrak{C}$ is principal and thus, since $m\equiv 2\pmod 4$, we can write 
$$
\mathfrak{C}=(a+b\sqrt{m}),
$$
where $a, b\in \mathbb{Z}$. 
Therefore, we have 
\begin{equation*}
a^2-mb^2=\pm 3.
\end{equation*}
We first look at the following equation
\begin{equation}\label{eq3.6}
a^2-mb^2=3.
\end{equation}
Let us assume $t=2n+1$. Then (\ref{eq3.6}) can be written as
\begin{equation}\label{eq3.7}
a^2\equiv 3\pmod {18t^2-1}.
\end{equation}
However, by quadratic reciprocity law, we see that
$$
\bigg( \frac{3}{18t^2-1}\bigg)=(-1)^\frac{18t^2-2}{2}\bigg(-\frac{1}{3}\bigg)=-1.
$$
This contradicts to (\ref{eq3.7}). 

We now look at the following:
\begin{equation}\label{eq3.8}
a^2-mb^2=-3.
\end{equation}
Clearly, $b\ne 0$. Let us suppose that (\ref{eq3.8}) has a solution in integers and without loss of generality, let us assume $(a_0, b_0)$ be an integer solution, with $b_0>0$ the least one. Then 
\begin{equation}\label{eq3.9}
a_0^2 - mb_0^2 = -3.
\end{equation}
In the norm form (\ref{eq3.9}) can be written as $N_{k_m}(\alpha)=-3$ with $\alpha=a_0-b_0\sqrt{m}$.
Let $l=6(2n+1)$. Then the fundamental unit $\epsilon_m$ in $k_m$ is given by 
$$
\epsilon_m=(l^2-1)+l\sqrt{m}.
$$
Thus $N_{k_m}(\alpha\epsilon_m)=- 3$ and this implies that
\begin{equation*}
(a_0(l^2-1)-b_0lm)^2-(b_0(l^2-1)-a_0l)^2m=\pm 3.
\end{equation*}
By the minimality of $b_0$, we obtain
\begin{equation*}
|b_0(l^2-1)-a_0l|\geq b_0.
\end{equation*}
If $-b_0(l^2-1)+a_0l\geq b_0$, then $a_0\geq b_0l$ and thus (\ref{eq3.9}) implies that $(l^2-m)b_0^2<-3$. This  is not possible. 

Finally, $b_0(l^2-1)-a_0l\geq b_0$ implies $b_0m > a_0l$ and hence (\ref{eq3.9}) gives $2mb_0^2<3l^2$ which implies $b_0=1$. Thus (\ref{eq3.9}) implies $a_0^2\equiv 3\pmod 4$ which is not true. This completes the proof.
\end{proof}
Applying Proposition \ref{proposition3.3} and Lemma \ref{lemma1}, we obtain the following result.
\begin{theorem}\label{theorem3.3}
Let $m=\{6(2n+1)\}^2-2$ with $n$ a positive integer. Then $\mathcal{H}^+(4m)>1$.
\end{theorem}

\section{Divisibility of $\mathcal{H}^+(m)$}
In this section, we prove a result concerning the divisibility of the plus part $\mathcal{H}^+(m)$ of the class numbers of certain cyclotomic fields. We first fix some notations. For a number field $K$, we denote the discriminant, the norm map and trace map of $K$ over $\mathbb{Q}$ by $D_K$, $N_{K}$ and $T_{K}$, respectively. For an integer $n$ and a prime $p$, by $\mathit{v}_p (n)$ we mean the greatest exponent $\mu$ of $p$ such that $p^\mu \mid n$.

Let us assume that $\alpha$ is an algebraic integer in $K$ such that $N_{K}(\alpha)$ is a cube in $\mathbb{Z}$. For such an $\alpha$, define the cubic polynomial $f_\alpha (X)$ by 
\begin{equation*}
 f_\alpha(X):=X^3-3(N_{K} (\alpha))^{1/3}X-T_{K}(\alpha).
\end{equation*}
Lemma 2.1 in \cite{HC-17} (or Lemma 2.2 in \cite{HC19}) and Proposition 6.5 in \cite{KI00} together give the following proposition which is one of the main ingredient in the proof of the next theorem. 
\begin{proposition}\label{proposition4.1}
Let $d=-3d'$ for some square-free integer $d' (\ne 1, - 3)$. Let $\alpha$ be an algebraic integer in $K'=\mathbb{Q}(\sqrt{d'})$ whose norm is a cube in $\mathbb{Z}$. Then the polynomial $f_\alpha(X)$ is irreducible over $\mathbb{Q}$ if and only if $\alpha$ is a not cube in $K'$. Moreover, if $f_{\alpha}(X)$ is irreducible over $\mathbb{Q}$, then the splitting field of $f_\alpha(X)$ over $\mathbb{Q}$ is an $S_3$-field containing $K=\mathbb{Q}(\sqrt{d})$ which is a cyclic cubic extension of $K$ unramified outside $3$ and contains a cubic subfield $L$ with $\mathit{v}_3(D_L)\ne 5$. 
\end{proposition}
We extract the following result from Theorem 1 in \cite{LN83} which talks about ramification at $p=3$.
\begin{proposition}\label{proposition4.2}
Let us suppose that
$$
g(X):= X^3-aX-b\in \mathbb{Z}[X]
$$
is irreducible over $\mathbb{Q}$ and that either $\mathit{v}_3(a)<2$ or $\mathit{v}_3(b)<3$ holds. Set $K:=\mathbb{Q}(\theta)$ for a root $\theta $ of $g(X)$. 
Then $3$ is totally ramified
 in $K/\mathbb{Q}$ if and only if one of the following conditions holds:
 \begin{enumerate}
  \item[(LN-1)] $1\leq \mathit{v}_3(b)\leq \mathit{v}_3(a)$;
  \item[(LN-2)] $3\mid a, \ a\not\equiv3(mod\ 9), \ 3\nmid b \ and \ b^2\not\equiv a+1 (mod\ 9)$;
  \item[(LN-3)] $ a \equiv3(mod\ 9), \ 3\nmid b \ and \ b^2\not\equiv a+1 (mod\ 27)$.
 \end{enumerate}
 \end{proposition}
Now we can proceed to our next result.
\begin{theorem}\label{theorem4.1}
For a positive integer $n$ satisfying $n\equiv 0 \pmod 3$, the class number of $K=\mathbb{Q}(\sqrt{3(4\times 3^n-1)})$ is divisible by $3$. In fact, there are infinitely many such real quadratic fields with class number divisible by $3$. 
\end{theorem}
\begin{proof}
Let $d=3(4\times 3^n-1)$ and $d'=1-4.3^n$. Also let $K'=\mathbb{Q}(\sqrt{d'})$. Suppose $\alpha \in K'$ is defined by 
 $$
 \alpha= \frac{1+\sqrt{d'}}{2}.
 $$
 Then $T_{k'}(\alpha)=1$ and $N_{K'}(\alpha)=3^{n/3}$. 
 
The cubic polynomial corresponding to $\alpha$ is
\begin{align*}
 f_{\alpha}(X)& =X^3-3(N_{K'}(\alpha))^{1/3}X-T_{K'}(\alpha)\\
  & = X^3-3^{(n+3)/3}X-1\\
  & \equiv X^3 - X-1 \pmod 2.
\end{align*}
Thus the polynomial $f_{\alpha} (X)$ is irreducible over $\mathbb{Z}_2$ and hence it is irreducible over $\mathbb{Q}$ too. Therefore by Proposition \ref{proposition4.1}, the splitting field of $f_{\alpha} (X)$ over $\mathbb{Q}$ is a cyclic cubic extension of $K$ which is unramified outside $3$.

We now claim that the splitting field of $f_{\alpha} (X)$ is unramified over $K$ at $3$ too. One can easily see that the polynomial $f_{\alpha} (X)$ does not satisfy the conditions (LN-1), (LN-2) and (LN-3). Therefore, by Proposition \ref{proposition4.2}, we prove the claim. Finally, by Hilbert class field theory the class number of $K$ is divisible by $3$.
\end{proof}
The following result is an immediate implication of Theorem \ref{theorem4.1} and Lemma \ref{lemma1}.
\begin{theorem}\label{theorem4.2}
Let $m=3(4\times 3^n-1)$ be square-free, where $n$ is a positive integer satisfying $n\equiv 0 \pmod 3$. Then $3\mid \mathcal{H}^+(m)$.
\end{theorem}
\section{Numerical examples}
In this section, we provide some numerical examples corroborating  our results in \S2 and in \S3. It is sufficient to compute the class numbers of each of the families of underlying real quadratic fields, i.e. $h(m)$'s. We compute these class numbers for  small values of $m$ and list them in the tables below.
\begin{table}[ht]
 \centering
\begin{tabular}{  c  c  c  c c c c c} 
 \hline
 $n$ & $m=\{14(2n+1)\}^2+2$ & $h(m)$ & $n$ & $m=\{14(2n+1)\}^2+2$ & $h(m)$\\
\hline
1&1766&5&2&4902&8\\
\hline
3&9606&18&4&15878&8\\
\hline
5&23718&12&6&33126&28\\
\hline
7&44102&15&8&56646&10\\
\hline
9&70758&5&10&86438&18\\
\hline
11&103686&32&12&122502&28\\
\hline
13&142886&33&14&164838&28\\
\hline
15&188358&44&16&213446&36\\
\hline
17&240102&7&18&268326&20\\
\hline
19&298118&21&20&329478&44\\
\hline
21&362406&4&22&396902&36\\
\hline
23&432966&66&24&470598&44\\
\hline
25&509798&55&26&550566&44\\
\hline
27&592902&21&28&636806&75\\
\hline
29&682278&68&30&729318&58\\
\hline
31&777926&57&32&828102&60\\
\hline
\end{tabular}\vspace{3mm}
\caption{Numerical examples of Theorem \ref{theorem3.1}. }
\end{table}
\begin{table}
 \centering
\begin{tabular}{  c  c  c  c c c c c } 
 \hline
 $n$ & $m=\{3(2n+1)\}^2+1$ & $h(m)$ & $n$ & $m=\{3(2n+1)\}^2+1$ & $h(m)$\\
\hline
1&82&4&2&226&8\\
\hline
3&442&8&4&730&12\\
\hline
5&1090&12&6&1522&12\\
\hline
7&2026&14&8&2602&10\\
\hline
9&3250&4&10&3970&20\\
\hline
11&4762&22&12&5626&28\\
\hline
13&6502&16&14&7570&20\\
\hline
15&8650&6&16&9802&2\\
\hline
17&11026&44&18&12322&20\\
\hline
19&13690&2&20&15130&32\\
\hline
21&16642&28&22&18226&36\\
\hline
23&19882&34&24&21610&48\\
\hline
25&23410&52&26&25282&32\\
\hline
27&27226&58&28&29242&38\\
\hline
29&31330&40&30&33490&48\\
\hline
\end{tabular}\vspace{2mm}
\caption{Numerical examples of Theorem \ref{theorem3.2}. }
\end{table}


\begin{table}
 \centering
\begin{tabular}{  c  c  c  c c  c c c } 
 \hline
 $n$ & $m=\{6(2n+1)\}^2-2$ & $h(m)$ & $n$ & $m=\{6(2n+1)\}^2-2$ & $h(m)$\\
\hline
1& 322&4&2&898&6\\
\hline
 3&1762&4&4&2914&12\\
 \hline
5&4354&16&6&6082&6\\
\hline
7&8098&12&8&10402&12\\
\hline
9&12994&24&10&15874&26\\
\hline
11&19042&24&12&22498&16\\
\hline
13&26242&12&14&30274&26\\
\hline
15&34594&4&16&39202&28\\
\hline
17&44098&16&18&49282&28\\
\hline
19&54754&24&20&60514&40\\
\hline
21&66562&24&22&72898&24\\
\hline
23&79522&48&24&86434&36\\
\hline
25&93634&56&26&101122&32\\
\hline
27&108898&40&28&116962&34\\
\hline
29&125314&44&30&133954&54\\
\hline
31&142882&40&32&152098&44\\
\hline
33&161602&8&34&171394&2\\
\hline
35&181474&60&36&191842&40\\
\hline
37&202498&52&38&213442&42\\
\hline
39&224674&60&40&236194&72\\
\hline
\end{tabular}\vspace{2mm}
\caption{Numerical examples of Theorem \ref{theorem3.3}. }
\end{table}
\begin{table}
 \centering
\begin{tabular}{  c  c  c  c c c c c} 
 \hline
 $n$ & $m=3(4\times3^n-1)$ & $h(m)$ & $n$ & $m=3(4\times3^n-1)$ & $h(m)$\\
\hline
3&321&3&6&8745&12\\
\hline
9&236193&36&12&6377289&36\\
\hline
15&172186881&837&18&4649045865&36\\
\hline
21&125524238433&11232&24&3389154437769&36\\
\hline
\end{tabular}\vspace{2mm}
\caption{Numerical examples of Theorem \ref{theorem4.1}. }
\end{table}
\section*{Acknowledgements}
A. Hoque acknowledges SERB MATRICS grant (No. MTR/2021/00762), Govt. of India.

\end{document}